\documentclass[a4paper,11pt]{amsart}

\usepackage[T1]{fontenc}
\usepackage[linktocpage]{hyperref}
\hypersetup{colorlinks=true,linkcolor=blue,citecolor=blue,filecolor=magenta,urlcolor=blue}      
\usepackage{geometry}       
\usepackage[english]{babel}  
\usepackage{amsthm,amsmath,amsfonts,amssymb}
\usepackage{tikz,tkz-graph,tikz-cd}
\usetikzlibrary{calc}
\usetikzlibrary{decorations.pathreplacing}
\usepackage{multirow}
\usepackage{booktabs}
\usepackage{pgfplots}

\newtheorem{theorem}{Theorem}[section]

\theoremstyle{definition}
\newtheorem{definition}[theorem]{Definition}

\theoremstyle{remark}
\newtheorem{rmk}[theorem]{Remark}

\newcommand{\CC}{\mathbb{C}}
\newcommand{\RR}{\mathbb{R}}

\newcommand{\QQ}{\mathbb{Q}}
\newcommand{\FF}{\mathbb{F}}
\newcommand{\PP}{\mathbb{P}}

\newcommand{\BB}{\mathbb{B}}

\newcommand{\A}{\mathcal{A}}

\newcommand{\q}{\mathbf{q}}
\newcommand{\Q}{\mathbf{Q}}

\renewcommand{\epsilon}{\varepsilon}

\begin{document}
	
	\title[LA-MOKA: Algorithm for the monodromy of line arrangements]{LA-MOKA: A Combinatorial Discretization Algorithm for the Braid Monodromy of Line Arrangements}
	
	\author[B. Guerville-Ball\'e]{Beno\^it Guerville-Ball\'e}
	\email{benoit.guerville-balle@math.cnrs.fr }
	\address{Matematica, Università di Sassari, via Vienna 2, 07100 Sassari, Italy}
	\thanks{}				
	

	\begin{abstract}
		Computing braid monodromy is a key tool for studying the topology of algebraic curves in the complex projective plane $\PP^2$. We present an algorithm, named LA-MOKA, to compute this invariant in the specific case of complex line arrangements. To safely manage the problem of floating-point approximations when working over number fields, we translate the continuous geometry into a discrete combinatorial structure. We establish explicit conditions on this discretization that guarantee the topological correctness of the computed braid monodromy.
	\end{abstract}
	
	\maketitle
	
	
	\section*{Introduction}
	
	The study of the topology of complex line arrangements in the complex projective plane $\PP^2$ is a classical subject in algebraic geometry, located at the intersection of the theories of algebraic plane curves and hyperplane arrangements. A central question in this field is to understand the relationship between the combinatorics of an arrangement, typically encoded in its intersection lattice, and the embedded topology of the arrangement $\A$, namely the homeomorphism type of the pair $(\PP^2,\A)$. It is known, for example, that the intersection lattice does not strictly determine the fundamental group of the complement $\PP^2\setminus\A$, see~\cite{Ryb, ACCM}. This makes the study and the computation of topological invariants necessary to distinguish arrangements that share the same combinatorics.
	
	The \emph{braid monodromy} is a complete invariant of the embedded topology of an arrangement $\A = \{L_1,\dots,L_n\}$; see~\cite{Car,ACG}. It is defined by considering a generic line $D$ and the associated affine projection $\pi:\PP^2\setminus D \rightarrow \CC$. Then, for each singular fiber of this projection, we consider a homotopical meridian $\mu$ in $\CC$ around the image of the fiber, and look at the trace of $\A$ in $\pi^{-1}(\mu)$. Such traces can be seen as braids in $\BB_n$. The set of all these braids forms a representative of the braid monodromy of $\A$. Several well-known topological invariants of $\A$ can be computed from it. For example, it allows to compute the fundamental group of $\PP^2\setminus\A$ using the Zariski-van Kampen presentation~\cite{Zar, vKam}, or various linking invariants more recently introduced~\cite{AFG,Rod}.
	
	\medskip
	
	The existence of arithmetic Zariski pairs~\footnote{~An \emph{arithmetic Zariski pair} is a pair of Galois-conjugated arrangements (which thus have isomorphic intersection lattices), but that have different topologies.} shows that purely algebraic methods cannot compute the braid monodromy. However, when it comes to computing it numerically, the problem of the approximation by floating-point numbers appears. For general algebraic plane curves, Marco and Rodríguez developed a certified algorithm implemented in SageMath~\cite{Sage} under the package SIROCCO\footnote{~Sirocco is included as an optional package since version 8 of SageMath.}~\cite{MarMar:Sirocco}. This approach is effective for general curves; however, when applied specifically to line arrangements, the computation time increases significantly. Since it is made for the general case, it does not take advantage of the restrictive geometry of line arrangements. The purpose of the present paper is to describe an algorithm, named LA-MOKA\footnote{~Stands for Line Arrangement Monodromy Obtained by $k$-grids Approximations.}, specifically made for line arrangements. As mentioned above, the main technical aspect in constructing this algorithm is to manage floating-number while ensuring the topological correctness of the resulting braids.
	
	The idea underlying our approach is to reduce the geometric situation to a combinatorial one. More precisely, for each singular fiber $D_i$, we want to safely approximate the actual trajectories of the intersection points of $\A\cap F$, when the fiber $F$ varies from a fixed generic fiber $D$ to the singular one $D_i$. In each of these singular fibers $D_i$ and in the fixed generic one, we construct a $k$-grid that encodes the relative positions of these intersection points. The grid is refined until it is sufficiently fine to isolate each intersection point in a distinct subsquare of the grid, and satisfies additional combinatorial conditions; see~(\ref{eq:min/max_non_crossing_condition}) and~(\ref{eq:last_non_crossing_condition}). This allows us to approximate the trajectory of  $L \cap F$ by the straight segment joining the centers of the subsquares containing $L\cap D$ and $L\cap D_i$; see Figure~\ref{fig:grids}.
	
	\medskip
	
	The paper is organized as follows. In Section~\ref{sec:monodromy}, we recall the general construction of the braid monodromy and show how it is used to compute the fundamental group  $\pi_1(\PP^2\setminus\A)$ of a line arrangement. In Section~\ref{sec:algo}, we describe our algorithm, followed in Section~\ref{sec:comparison} by a comparison between our algorithm LA-MOKA and the well-established package SIROCCO. Finally, in Section~\ref{sec:proofs}, we provide the proofs of two specific technical points.
	
	\bigskip
	
	\noindent \textbf{Notation.} Throughout the paper, $\A = \{L_1 , \dots , L_n\}$ denotes a line arrangement in $\PP^2$, and $\Q = \{Q_1, \dots , Q_m\}$ denotes its set of singular points. The set $\{1,\dots,n\}$ is denoted by $[n]$.
	
	\section{The braid monodromy and fundamental group}\label{sec:monodromy}
	
	In this section, we recall the construction of the \emph{braid monodromy} of a line arrangement $\A$ in $\PP^2$. For a detailed and more general construction, we refer to~\cite{CohSuc:monodromy}. We then use this braid monodromy to compute presentations of the fundamental group $\pi_1(\PP^2\setminus\A)$ of the arrangement~$\A$. 
	
	\subsection{Construction of the braid monodromy}
	
	Let $D$ be a line of $\PP^2$ generic with respect to $\A$, and let $P_0$ be a fixed point of $D \setminus \A \cap D$. We assume that $P_0$ is such that for each $Q_i \in \Q$, the only point of $\Q$ contained in the line $D_i$ passing through $Q_i$ and $P_0$ is $Q_i$ itself. The data of $D$ and $P_0$ defines a projection $\pi : \PP^2 \setminus P_0 \to \PP^1$, which can be restricted to a projection $\pi_{\mid \CC^2} : \PP^2\setminus D \simeq \CC^2 \to \CC$. We denote by $\q = \{q_i := \pi(Q_i) \mid Q_i \in \Q\}$ the images of the points of $\Q$ under $\pi$. Note that the assumption on $D$ and $P_0$ implies that the $q_i$ are pairwise distinct.
	
	Let $p_0$ be the image of $D\setminus P_0$ under $\pi$, which we fix as a base point on $\PP^1$. Let $\gamma$ be an element of $\pi_1(\PP^1 \setminus \q , p_0)$. For all $t\in [0,1]$, the $\pi$-fiber over $\gamma(t)$ intersects the arrangement $\A$ generically, and satisfies $\pi^{-1}(\gamma(0)) = \pi^{-1}(\gamma(1)) = D\setminus P_0$. As a consequence, $\gamma$ induces an isotopy $h_\gamma$ of $\CC \simeq D \setminus P_0$ fixing the elements of $D \cap \A$ pointwise. Such isotopies are elements of the mapping class group of $(\CC , D\cap \A)$. As such, they can be viewed as braids in the braid group $\BB_n$ on $n$ strands. This defines a morphism $\phi_{(\A,D,P_0)} : \pi_1(\PP^1 \setminus \q, p_0) \to \BB_n$, often denoted simply $\phi_\A$.
	
	\medskip
	
	From an abstract point of view, $\pi_1(\PP^1\setminus \q, p_0)$ is isomorphic to $\FF_{m-1}$ the free group of rank~$m-1$. Nevertheless, we want to keep track of the geometric nature of this group. For $i\in [m]$, let $\gamma_i$ be a \emph{meridian} around $q_i$ based on $p_0$, that is to say: it starts at $p_0$, goes to a point near $q_i$ along a path $s_i$, loops around $q_i$, and returns to $p_0$ along $s_i$ in reverse. If $(\gamma_1, \dots, \gamma_m)$ satisfies the condition that the reverse product $\gamma_m \dots \gamma_1$ is trivial in $\pi_1(\PP^1\setminus \q, p_0)$ then it is called a \emph{geometric basis of meridians} in $\pi_1(\PP^1\setminus \q, p_0)$.
	
	\begin{definition}
		The braid monodromy associated with $(\A,D,P_0)$ is the $m$-tuple of braids:
		\begin{equation*}
			(\phi_\A(\gamma_1), \dots, \phi_\A(\gamma_m) ) \in (\BB_n)^m,
		\end{equation*}
		where $\gamma_1,\dots,\gamma_m$ is a geometric basis of meridians in $\pi_1(\PP^1\setminus \q, p_0)$.
	\end{definition}
	
	\begin{rmk}
		The braid monodromy associated with $(\A,D,P_0)$ is not well-defined. It depends on the ordering of the singular points and on the choice of the geometric basis. Nevertheless, it is well defined up to the Hurwitz action of $\BB_m$ on $(\BB_n)^m$ given by:
		\begin{equation*}
			(\phi_\A(\gamma_1), \dots, \phi_\A(\gamma_m) )^{\sigma_i} = (\phi_\A(\gamma_1), \dots, \phi_\A(\gamma_{i+1}), \phi_\A(\gamma_{i+1})\phi_\A(\gamma_{i})\phi_\A(\gamma_{i+1})^{-1}, \dots, \phi_\A(\gamma_m) ).
		\end{equation*}
		Additionally, changing the generic line $D$ or the projection point $P_0$ modifies the terms of the braid monodromy by conjugating them all by the same braid in $\BB_n$. Thus, the braid monodromy of an arrangement $\A$ (independently of the chosen projection) is well-defined up to Hurwitz action and simultaneous conjugation.
	\end{rmk}
	
	Since the meridian $\gamma_i$ is defined as a path $s_i$ from $p_0$ to a nearby point of $q_i$, a loop around $q_i$, and then the path $s_i$ backward, we can describe the braid $\phi_\A(\gamma_i)$ as the conjugate $b_i  (\Delta_i)^2  b_i^{-1}$, where $b_i$ is the braid over the path $s_i$, and $(\Delta_i)^2$ is a local full-twist of the strings corresponding to the lines of $\A$ passing through $Q_i$, see Figure~\ref{fig:monodromy}. When the braid monodromy is given as such conjugates, we call it a \emph{decomposed braid monodromy}. It is useful, for example, to compute Arvola's presentation of the fundamental group~\cite{Arv}.
	
	\begin{figure}[!ht]
		\begin{tikzpicture}
			\begin{scope}[scale = 0.66]
				\begin{scope}[xscale = 0.5, yslant=0.6]
					\draw (0,0) -- (0,6) -- (6,6) -- (6,0) -- (0,0);
					\node[right] at (6,6) {$D$};
					
					\node (P1) at (3.5,1) {};
					\node[left] () at (P1) {$L_4$};
					\node (P2) at (2.5,2.5) {};
					\node[left] () at (P2) {$L_3$};
					\node (P3) at (3,3) {};
					\node[left] () at (P3) {$L_2$};
					\node (P4) at (2.5,4.5) {};
					\node[left] () at (P4) {$L_1$};
					
					\node (b3) at (3,6) {};
					
					\draw[dotted] (0,0) -- (2,-3) -- (6,0);
					\node (p0) at (2,-3) {$\bullet$};
					\node[left] () at (2,-3) {$p_0$};
				\end{scope}
				
				\begin{scope}[xshift = 204, xscale = 0.5, yslant=0.6]
					\draw (0,0) -- (0,6) -- (6,6) -- (6,0) -- (0,0);
					\node (c1) at (0,0) {};
					\node (c2) at (0,6) {};
					\node[right] at (6,6) {$D_i$};
					
					\draw[dotted] (0,0) -- (2,-3) -- (6,0);
					\node (qi) at (2,-3) {$\bullet$};
					\node[below] () at (2,-3.25) {$q_i$};
					
					\node (Q1) at (3,3) {};
					\node (Q2) at (3,1) {};
					
					\node (r) at (9,0) {$\PP^2\setminus P_0$};
					\node (s) at (9,-6) {$\PP^1$};
					\draw[->>] (r) -- (s) node[pos = 0.5, right] {$\pi$};
				\end{scope}
				
				\begin{scope}[xshift = 190, xscale = 0.5, yslant=0.6]
					\draw[dashed] (0,0) -- (0,6) -- (6,6) -- (6,0) -- (0,0);
					\node (d1) at (0,0) {};
					\node (d2) at (0,6) {};
					
					\node (d3) at (3,6) {};
					
					\draw[dotted] (0,0) -- (2,-3) -- (6,0);
					\node (almost_qi) at (2,-3) {};
				\end{scope}
				
				\draw[line width=4pt, white] (P1.center) -- ($(P1.center)!4/5!(Q1.center)$r);
				\draw[thick] (P1.center) -- (Q1.center);
				\draw[line width=4pt, white] (P3.center) -- ($(P3.center)!4/5!(Q1.center)$);
				\draw[thick] (P3.center) -- (Q1.center);
				
				\draw[line width=4pt, white] (P4.center) -- (Q2.center);
				\draw[thick] (P4.center) -- (Q2.center);
				
				\draw[white, line width=4pt] ($(P2.center)!1/5!(Q1.center)$) -- ($(P2.center)!4/5!(Q1.center)$);
				\draw[thick] (P2.center) -- (Q1.center);
				
				\node[above right] () at (Q1) {$Q_i$};
				
				\draw[white, line width=4pt] ($(c1.center)!1/10!(c2.center)$) -- ($(c1.center)!9/10!(c2.center)$);
				\draw (c1.center) -- (c2.center);
				
				\draw[white, line width=4pt] ($(d1.center)!1/10!(d2.center)$) -- ($(d1.center)!9/10!(d2.center)$);
				\draw[dashed] (d1.center) -- (d2.center);
				
				\draw (p0.center) -- (almost_qi.center) node[pos=0.5,above] {$\gamma_i$};
				\draw ($(almost_qi.center)+(0.5,0)$) ellipse [x radius=0.5, y radius=0.25];
				
				\draw[decorate, decoration={brace, amplitude=5pt, raise=3pt}] ($(b3.east) - (0.05,0)$) -- node[above=8pt]{$b_i$} ($(d3.east) - (0.05,0)$);
				
				\node () at (Q1.center) {$\bullet$};
				
			\end{scope}
		\end{tikzpicture}
		\caption{Construction of the braid monodromy. \label{fig:monodromy}}
	\end{figure}
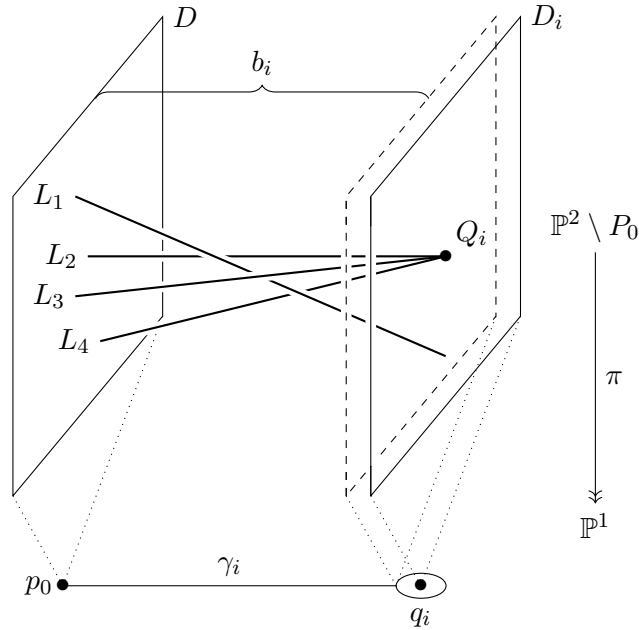
	
	\subsection{Presentations of the fundamental group}
	
	The braid group $\BB_n$ acts naturally on $\pi_1(\CC \setminus \{p_1,\dots,p_n\}) \simeq \FF_n$ as follows: for a generator $\mu_i$ of $\FF_n$, and the standard generator $\sigma_j$ of $\BB_n$ interchanging strands $j$ and $j+1$:
	\begin{equation*}
		\mu_i^{\sigma_j} =
		\left\{
		\begin{array}{ll}
			\mu_i \cdot \mu_{i+1} \cdot \mu_i^{-1} & \text{if } i= j, \\
			\mu_{i-1} & \text{if } i = j+1, \\
			\mu_i & \text{otherwise.}
		\end{array}
		\right.
	\end{equation*}
	Using this action, we can describe the so-called Zariski-van Kampen presentation of the fundamental group of $\PP^2\setminus\A$~\cite{Zar,vKam}; see also~\cite{Cog}.
	
	\begin{theorem}[\cite{Zar,vKam}]
		Let $\A$ be a line arrangement in $\PP^2$, and let $(\beta_1,\dots,\beta_m)$ be a representative of the braid monodromy of $\A$. We have the following presentation:
		\begin{equation*}
			\pi_1(\PP^2\setminus\A) \simeq \langle\  \mu_1, \cdots ,\mu_n \mid \mu_1 \cdots \mu_n = 1, \text{ and }  \mu_i = \mu_i^{\beta_j}\quad \text{for }i\in [n], j\in [m] \  \rangle.
		\end{equation*}
	\end{theorem}
	
	When it comes to computing a presentation of the fundamental group of a line arrangement explicitly, the Zariski-van Kampen one is often too costly. Indeed, the computation of the action of the braids $\beta_i$ on the free group $\FF_n$ provide quickly extremely long words. To shorten the words considered, and thus reduce the required computation time, we can rely on Arvola's presentation~\cite{Arv}, first developed by Randell for complexified real arrangements~\cite{Ran, Ran:correction}. Usually, the computation of this presentation is made using the \emph{braided wiring diagram}, which is a derived version of the braid monodromy. In this paper, we present a slightly different formulation of this presentation based on the decomposed braid monodromy\footnote{~The arguments given by Arvola in~\cite{Arv} justify in the same way the current formulation.}.
	
	\medskip
	
	For elements $a_1 , \dots , a_l$ in $\FF_n$, we define their \emph{generalized commutator}, denoted $[a_1,\dots,a_l]$, as the following set of relations where $\sigma$ runs over the cyclic permutations on $l$ elements.
	\begin{equation*}
		[a_1,\dots,a_l] = \{a_1 \cdots a_l = a_{\sigma(1)} \cdots a_{\sigma(l)} \}.
	\end{equation*}
	Note that this generalized commutator is sensitive to the order of the elements $a_1,\dots,a_l$. 
	
	Assume that we have a decomposed braid monodromy $(b_1.(\Delta_1)^2.b_1^{-1},\dots,b_m.(\Delta_m)^2.b_m^{-1})$. In the base-fiber $D\setminus P_0 \simeq \CC$, the lines of $\A$, and so the meridians $\mu_1,\dots,\mu_n$, are ordered according to the imaginary parts of the points $(D\setminus P_0) \cap L_i$, for $i\in [n]$. Through the usual morphism $\BB_n\rightarrow S_n$, the braid $b_i$ determines a permutation $\varsigma_i$. This permutation acts on the ordering of the lines in $D\setminus P_0$. In particular, it induces an ordering on the lines passing through $Q_i$. Let $Q_i \in \Q$ be a singular point of multiplicity $m_i$ and defined as the intersection of $L_{i_1}, \cdots, L_{i_{m_i}}$. Assume that the local ordering induced by $\varsigma_i$ is:
	\begin{equation*}
		L_{i_1} <_i L_{i_2} <_i \cdots <_i L_{i_{m_i}}.
	\end{equation*}
	If $(b_1.(\Delta_1)^2.b_1^{-1},\dots,b_m.(\Delta_m)^2.b_m^{-1})$ is a decomposed braid monodromy of $\A$, we define the following set of relations in $\FF_n = \langle \mu_1, \dots, \mu_n \rangle$:
	\begin{equation*}
		R_i = [\mu_{i_1}^{b_i}, \dots , \mu_{i_{m_i}}^{b_i}].
	\end{equation*}
	
	\begin{theorem}[\cite{Arv}]
		Let $\A$ be a line arrangement of $\PP^2$, and $(b_1.(\Delta_1)^2.b_1^{-1},\dots,b_m.(\Delta_m)^2.b_m^{-1})$ be a decomposed braid monodromy of $\A$. We have the following presentation:
		\begin{equation*}
			\pi_1(\PP^2\setminus\A) \simeq \langle\  \mu_1, \ldots ,\mu_n \mid \mu_1 \cdot \ldots \cdot \mu_n, \text{ and }\  R_i,\ i\in [m]  \rangle.
		\end{equation*}
	\end{theorem}

	\section{The algorithm LA-MOKA}\label{sec:algo}
	
	The algorithm presented below enable the computation of a decomposed braid monodromy of any line arrangement $\A$. It is made to be safe to use on a computer since it replaces the unsafe floating-point numbers used to approximate elements of a number field with integers. We denote by $\FF$ the field of definition of $\A$, that is the smallest subfield of $\CC$ containing the coefficients of all the lines of $\A$. Most of the time, it will be $\QQ$ or a number field. 
	
	\mbox{}
	
	\noindent \textbf{Step 1.} Find a generic projection.
	
	Take a random projective transformation over $\FF$ and apply it to $\A$. To keep the notation simple, the image of $\A$ will also be denoted by $\A$. Take a random line $D$ with coefficients in $\QQ[i]\setminus\{0\}$ if $\FF$ is totally real, and in $\QQ\setminus\{0\}$ otherwise\footnote{~We can take $D$ with coefficients in $\QQ[i]\setminus\{0\}$ even if $\FF$ is not totally real, but if $\FF$ does not already contain $i$, it makes things unnecessarily complicated.}. We denote by $P_0$ the intersection of $D$ with the line $z=0$. The resulting arrangement $\A$ and line $D$ must satisfy the following conditions:
	\begin{enumerate}
		\item The line $z=0$ is not a line of the arrangement $\A$.
		\item The line $D$ is generic with respect to $\A$, i.e., $D\cap\A$ is composed of $n$ distinct points.
		\item For each $Q_i \in \Q$, the line $D_i$ passing through $Q_i$ and $P_0$ intersects $\Q$ only at $Q_i$ itself.
		\item For all $i \in [m]$, the only point of $\q$ contained in the segment $[p_0,q_i]$ is $q_i$.
	\end{enumerate}
	Remark that all the conditions required are Zariski-open, so the random choices of the projective transformation and the line $D$ guarantee that the conditions are almost always satisfied. If this is not the case, you are unlucky and should repeat Step~1.
	
	\begin{rmk}
		We acknowledge that relying on randomness for genericity may seem counterintuitive in a certified method. If we want a deterministic search for the genericity, we can take an infinite family of projective transformations and of lines $D$, then try them one by one (as done in SIROCCO for example). The Zariski-open nature of these conditions guarantees, after a finite number of attempts, the genericity of the projective transformation and of the line $D$. While this may seem more rigorous, the initial choice of the infinite family is made arbitrarily by the person who implements the algorithm. Also, while relying on randomness, it is extremely likely that on the first attempt the genericity is already reached, while for the deterministic search this might take several attempts. So, to decrease the runtime of our algorithm, we decide to use randomness.
	\end{rmk}
	
	\mbox{}
	
	\noindent \textbf{Step 2.} Construction of good $k$-grids for the fibers $D$ and $D_i$.
	
	The idea of this step is to find, for each $F \in \{D,D_1,\dots,D_m\}$, a square $S_F$ subdivided into $k^2$ subsquares\footnote{~The value of $k$ should be the same for all $F\in\{D,D_1,\dots,D_m\}$.} such that all the intersections of $\A \cap F$  are contained in distinct subsquares of $S_F$. Moreover, these $k$-grids should be fine enough to enable the approximation of the points of $\A\cap F$ by the centers of the corresponding subsquares of the grid.
	
	\medskip
	
	Let $F\in \{D,D_1,\dots,D_m\}$. Since the line $z=0$ is not a line of $\A$, all the points of $F\cap\A$ are in the chart $z \neq 0$. So, we can identify $F\setminus P_0$ with $\CC$ through the map $p_F : (x:y:z) \mapsto x/z$. This is an isometry since all the coefficients of $D$ are non-zero.
	
	Consider $S_F$ a square of $\CC$ that is aligned with the axes of $\CC$, and that contains all the points of $p_F(F\cap\A)$. By increasing $k$, we can assume that the points of $p_F(F\cap \A)$ are in distinct subdivisions of the square $S_F$. We define the \emph{grid-map} as $g_F:\A\to [k]^2$ that associates with each line $L \in \A$ the indices of the subsquare of $S_F$ containing the point $p_F(F\cap L)$. The $(0,0)$ index is associated with the subsquare with the smallest real and imaginary parts. Note that the images under the grid-map $g_{D_i}$ of all lines passing through $Q_i$ are equal.
	
	\medskip
	
	For each $Q_i \in \Q$, and any pair of line $L_u$ and $L_v$ in $\A$, we define
	\begin{equation}\label{eq:def_abcd}
		\delta(u,v) = g_D (L_u) - g_D(L_v) \quad\text{and}\quad \delta_i(u,v) = g_{D_i}(L_u) - g_{D_i}(L_v).
	\end{equation}
	Remark that $\delta$ and the $\delta_i$ are skew-symmetric, i.e., $\delta(u,v) = - \delta(v,u)$ and $\delta_i(u,v) = - \delta_i(v,u)$.To lighten notation below, when $L_u$, $L_v$, and $i$ are fixed, we write $\delta(u,v) = (a,b)$ and $\delta_i(u,v) = (c,d)$. For example, instead of $\delta_i(u,v)_1$ we write $c$. The conditions required for the grid-maps are: for all $Q_i\in\Q$, and for all $L_u$ and $L_v$ in $\A$ such that $L_u\neq L_v$ and $L_u\cap L_v \neq Q_i$,
	\begin{enumerate}
		\item The values $a$, $b$, $c$ and $d$ are all non-zero. In other words, the subsquares containing the images of a point of $F\cap \A$ are in distinct rows and distinct columns.
		\item At least one of the following five inequalities is satisfied:
		\begin{equation}\label{eq:min/max_non_crossing_condition}\tag{C1}
			\min(a, c) \geq 1, \quad \max(a, c) \leq -1, \quad \min(b, d) \geq 1, \quad \max(b, d) \leq -1,
		\end{equation}
		\begin{equation}\label{eq:last_non_crossing_condition}\tag{C2}
			| ad - bc | > |a - c| + |b - d|.
		\end{equation}
	\end{enumerate}
	
	\medskip
	
	The randomness of $D$ ensures that almost always the points of $p_F(F\cap\A)$ have pairwise distinct real (resp. imaginary) parts. If this is not the case, we return to Step~1; otherwise, for sufficiently large $k$, the first condition will be satisfied. The second condition will be discussed in Section~\ref{subsec:condition}, and we will see that, here again, it is satisfied for $k$ sufficiently large.
	
	\begin{rmk}
		While iterating over $k$, ensure that the distance between the points $p_F(F\cap\A)$ and the grid boundaries is larger than the floating-point precision. This can be done, for example, by checking that for a point $x = L\cap F$ and $\epsilon \ll 1$, the four points $p_F(x\pm \epsilon)$ and $p_F(x\pm i\epsilon)$ are all contained in the same subsquare of the grid.
	\end{rmk}
	\mbox{}
	
	\noindent \textbf{Step 3.} Computation of the braid associated to $Q_i$.
	
	To compute the braid $\beta_i$ associated with $Q_i$, we consider the meridian $\gamma_i$ around $q_i$ whose path $s_i$ is contained in the segment joining $p_0$ to $q_i$. By Condition~(4) of Step~1, the path $s_i$ does not intersect any other $q_j$, with $j\neq i$. Since $s_i$ is on a straight line, the intersection of a line $L_u$ with the preimage of $s_i$ under $\pi$ is the segment $\ell_u^i$ joining $D\cap L_u$ to $D_i\cap L_u$. In Step~2, we compute two squares $S_D$ and $S_{D_i}$. We can lift them to $D$ and $D_i$ respectively, so all the segments considered are in the convex hull of $p_D^{-1}(S_D)$ and $p_{D_i}^{-1}(S_{D_i})$, which can be identified with a cube. This cube comes with two $k$-grids on two of its opposite faces, the ones included in $D$ and $D_i$. Conditions~(\ref{eq:min/max_non_crossing_condition}) and~(\ref{eq:last_non_crossing_condition}) ensure that we can approximate the segment $\ell_u^i$ by the one joining the center of the subsquares that contain the ends of $\ell_u^i$. See Figure~\ref{fig:grids}, the solid segment is $\ell_u^i$ and the dashed one is its approximation.
	
	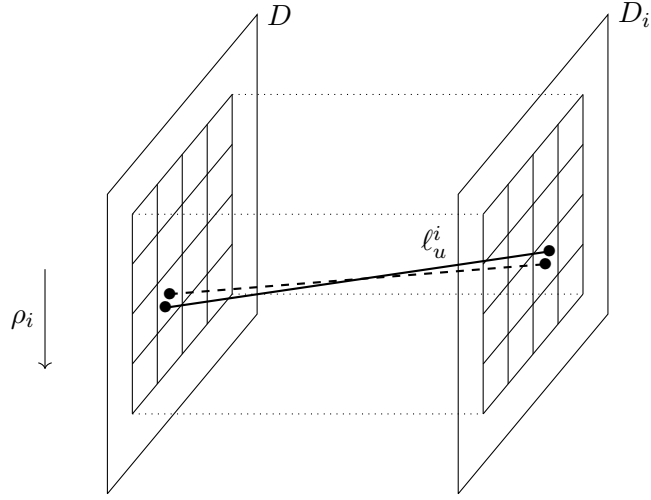
\begin{figure}[!ht]
		\begin{tikzpicture}
			\begin{scope}[scale = 0.66]
				\begin{scope}[xscale = 0.5, yslant=0.6]
					\draw[->] (-2.5,6) -- (-2.5,4) node[pos=0.5,left] {$\rho_i$};
					\draw (0,0) -- (0,6) -- (6,6) -- (6,0) -- (0,0);
					\node[right] at (6,6) {$D$};
					\foreach \i in {1,...,5} {
						\draw (1,\i) -- (5,\i);
						\draw (\i,1) -- (\i,5);
					}
					\node (P1) at (1,1) {};
					\node (P2) at (1,5) {};
					\node (P3) at (5,5) {};
					\node (P4) at (5,1) {};
					
					\node (P) at (2.33,2.33) {$\bullet$};
					\node (p) at (2.5, 2.5) {$\bullet$};
				\end{scope}
				\begin{scope}[xshift = 200, xscale = 0.5, yslant=0.6]
					\draw (0,0) -- (0,6) -- (6,6) -- (6,0) -- (0,0);
					\node[right] at (6,6) {$D_i$};
					\foreach \i in {1,...,5} {
						\draw (1,\i) -- (5,\i);
						\draw (\i,1) -- (\i,5);
					}
					\node (Q1) at (1,1) {};
					\node (Q2) at (1,5) {};
					\node (Q3) at (5,5) {};
					\node (Q4) at (5,1) {};
					
					\node (Q) at (3.66,2.66) {$\bullet$};
					\node (q) at (3.5,2.5) {$\bullet$};
				\end{scope}
				\draw[dotted] (P1.center) -- (Q1.center);
				\draw[dotted] (P2.center) -- (Q2.center);
				\draw[dotted] (P3.center) -- (Q3.center);
				\draw[dotted] (P4.center) -- (Q4.center);
				\draw[thick] (P.center) -- (Q.center) node[pos=0.7,above] {$\ell_u^i$};
				\draw[dashed,thick] (p.center) -- (q.center);
			\end{scope}
		\end{tikzpicture}
		\caption{Grids in the fibers $D$ and $D_i$, and the approximation of $\ell_u^i$. \label{fig:grids}}
	\end{figure}
	
	Instead of doing the computation in the preimage of $s_i$ under $\pi$ and with the segments $\ell_u^i$, we can consider the situation where the ambient space is $\RR^3$ and the centers of the subsquares of the grids of $S_D$ and $S_{D_i}$ are respectively $(0,r,s)$ and $(1,r,s)$, for $r,s\in [k]$. Thus, the coordinates of the endpoints of the segments considered are given by the grid-maps $g_D$ and $g_{D_i}$.
	
	To describe the braid $\beta_i$, we take the projection $\rho_i$ onto the $x$ and $y$ coordinates, shown in Figure~\ref{fig:grids} by the arrow on the left, and we will describe the sign of the crossings. Note that due to Condition~(1) of Step~2, the endpoints of the segments project onto distinct points, except for those corresponding to the point $Q_i$. Nevertheless, the situation is no longer generic, so it is possible that the projection of three or more segments intersect at a single point. We will see later how to handle such a situation.
	
	Since all the computations are performed for a fixed $i\in[m]$, we omit this index from the notation. For $L_u$ and $L_v$ in $\A$ such that $L_u \cap L_v \neq Q_i$, we denote $l_u$ and $l_v$ their associated segments in $\RR^3$, and as in Step~2, we set $\delta(u,v)=(a,b)$ and $\delta_i(u,v)=(c,d)$. The segments $\rho(l_u)$ and $\rho(l_v)$ intersect if and only if $ac < 0$. Such a crossing appears when the first coordinate is $t_{u,v} = a / (a-c)$, and the sign of the crossing is the sign of $ad-cb$, which we denote by $e_{u,v}$.
	
	\medskip
	
	To explicitly give the braid word of $b_i$ in terms of the $\sigma_j$, the usual generators of the braid group $\BB_n$, we first label the strands according to their initial positions. That is, we create a vector $\nu = (\nu_1,\cdots,\nu_n)$ whose entries are the elements of $[n]$ and such that:
	\begin{equation*}
		g_D(L_{\nu_1})_1 < g_D(L_{\nu_2})_1 < \cdots < g_D(L_{\nu_n})_1,
	\end{equation*}
	where the subscript $1$ denotes the first coordinate of the $g_D(L_{\nu_j})$. Remark that this initial vector $\nu$ will be the same for all $i\in [m]$. Then, we order increasingly and according to their time of crossing $t_{u,v}$, the pairs $(u,v)$ forming a crossing. We initialize with the trivial braid $b_i = 1_{\BB_n}$, and run over the pairs $(u,v)$. We denote by $r$ the minimal index in $\nu$ between $u$ and $v$\footnote{~Remark that by construction $u$ and $v$ are adjacent in $\nu$.}, and we add to $b_i$ the term $\sigma_r^{e_{u,v}}$ if the index of $u$ in $\nu$ is smaller than the one of $v$, and $\sigma_r^{e_{v,u}}$ otherwise. Then, we permute in $\nu$ the position of $u$ and $v$. The braid $b_i$ obtained at the end of the process, is the braid $b_i$ of the decomposed braid monodromy. To get $\beta_i$, we need to consider the local full twist $(\Delta_i)^2$ associated to the singular points $Q_i$. Its expression in terms of the $\sigma_j$ is:
	\begin{equation*}
		\Delta_i = (\sigma_{s}\cdots\sigma_{s+m_i})(\sigma_{s}\cdots\sigma_{s+m_i-1})\cdots(\sigma_{s}),
	\end{equation*}
	with $m_i$ is the multiplicity of $Q_i$, and $s$ is the minimum among the indices in $\nu$ of $u_1,\dots,u_{m_i}$, where $L_{u_1},\dots,L_{u_{m_i}}$ are all the lines passing through $Q_i$.
	
	As noted before, it is possible that several pairs $(u,v)$ form a crossing at the same time $t_{u,v}$. To handle this, we proceed according to adjacency in $\nu$. Regardless of the order considered, the resulting braids $b_i$ will be equivalent, see Section~\ref{subsec:multiple_intersection}.
	
	\medskip
	
	For each $Q_i$, we obtain a braid $\beta_i = b_i.(\Delta_i)^2.b_i^{-1}$. A decomposed braid monodromy of $(\A,D,P_0)$ is then
	\begin{equation*}
		(\beta_1,\cdots,\beta_m) \in (\BB_n)^m.
	\end{equation*}
	
	\begin{rmk}
		A SageMath implementation of this algorithm is available at:
		\begin{center}
			\url{https://github.com/guerville-balle/LA-MOKA}
		\end{center}
	\end{rmk}

	\section{Performance comparison with SIROCCO}\label{sec:comparison}
	
	In this section, we compare the performance of the LA-MOKA algorithm presented in this paper with SIROCCO, the standard library developed by Marco and Rodríguez~\cite{MarMar:Sirocco}. As mentioned in the introduction, the algorithm implemented in SIROCCO is a general tool for computing braid monodromy of algebraic plane curves. As such, it is not optimized for the very specific case of line arrangements. In contrast, the LA-MOKA algorithm is specifically designed for line arrangements and exploits its geometry. Obviously, the counterpart of this specificity is that LA-MOKA does not work for general algebraic curves.
	
	To compare the efficiency of LA-MOKA and SIROCCO, we conducted a benchmark\footnote{~All the computations have been performed on a personal laptop with a CPU: 12th Gen Intel(R) Core(TM) i5-1230U, and 8Go of RAM.} on line arrangements with $n$ lines, for $n\in\{5,6,7,8,9\}$. To avoid, as much as possible, bias in the selection of the examples, we ran tests over all the line-combinatorics with a non-empty realization space. For each one, we selected $10-n$ representatives and compute its braid monodromy using both algorithms. We also recorded, for each $n$, of the maximal execution time. All results are summarized in Table~\ref{tab:benchmark}, with execution times expressed in seconds.
	
	\begin{table}[h!]
		\centering
		\renewcommand{\arraystretch}{1.2} 
		\begin{tabular}{|c|cc|cc|}
			\hline
			\multirow{2}{*}{\textbf{Lines ($n$)}} & \multicolumn{2}{c|}{\textbf{SIROCCO}} & \multicolumn{2}{c|}{\textbf{LA-MOKA}} \\
			\cline{2-5}
			& \textbf{Avg Time (s)} & \textbf{Max Time (s)} & \textbf{Avg Time (s)} & \textbf{Max Time (s)} \\
			\hline
			5 & 0.4131 & 0.4523 & 0.0508 & 0.0633 \\
			6 & 0.6824 & 0.9141 & 0.0536 & 0.0675 \\
			7 & 1.4921 & 2.2297 & 0.0853 & 0.1131 \\
			8 & 3.3134 & 26.3434 & 0.1254 & 0.1933 \\
			9 & 9.8892 & 500.2286 & 0.1923 & 0.5792 \\
			\hline
		\end{tabular}
		\vspace{0.5cm}
		\caption{Performance comparison between SIROCCO and LA-MOKA for line arrangements with $n$ lines, for $n\in\{5,6,7,8,9\}$.\label{tab:benchmark}}
	\end{table}
	
	The benchmark data highlights two clear phenomena:
	\begin{enumerate}
		\item \textbf{Runtime growth}: While the average runtime of SIROCCO grows rapidly from $0.41$ seconds at $n=5$ to nearly $10$ seconds at $n=9$, the average execution time of LA-MOKA grows minimally, staying well below $0.20$ seconds for $n=9$.
		\item \textbf{Variability}: The maximum observed runtime for SIROCCO of $26.3$ seconds at $n=8$ (corresponding to $7.9$ times the average) and over $500$ seconds at $n=9$ ($\sim 8.3$ minutes, or $50.5$ times the average), illustrates the high variability of this algorithm. In contrast, the maximum runtime for LA-MOKA at $n=9$ remains under $0.58$ seconds (about $3$ times the average), demonstrating high predictability.
	\end{enumerate}
	
	The evolution of the average runtimes is depicted in Figure~\ref{fig:benchmark_graph}. The plot illustrates the low growth of LA-MOKA relative to the steep growth curve of SIROCCO. By replacing root continuation with discrete $k$-grids, LA-MOKA algorithm provides a significant performance gain specifically tuned for line arrangements.
	
	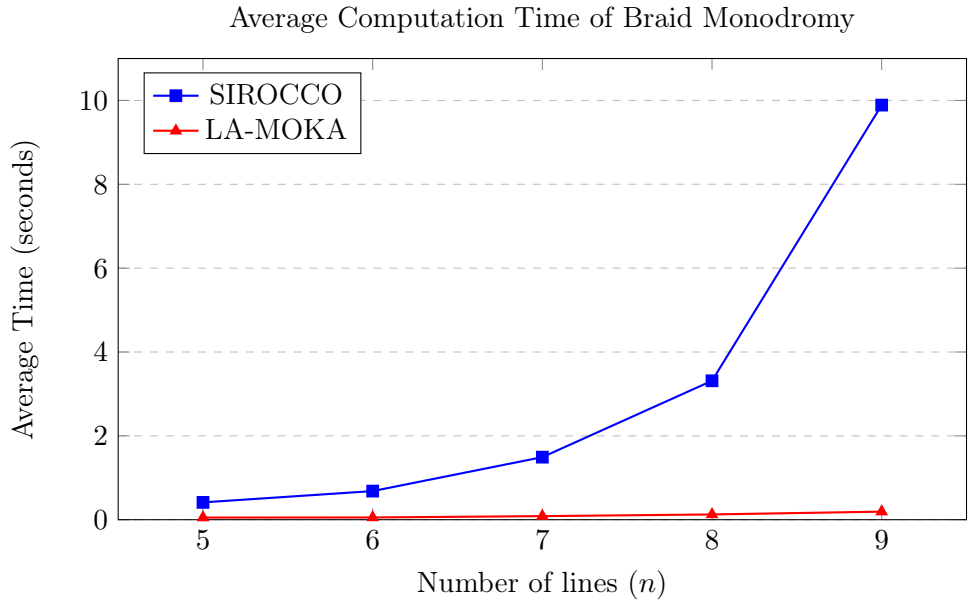
\begin{figure}[h!]
		\centering
		\begin{tikzpicture}
			\begin{axis}[
				title={Average Computation Time of Braid Monodromy},
				xlabel={Number of lines ($n$)},
				ylabel={Average Time (seconds)},
				xmin=4.5, xmax=9.5,
				ymin=0, ymax=11,
				xtick={5, 6, 7, 8, 9},
				legend pos=north west,
				ymajorgrids=true,
				grid style=dashed,
				width=0.8 \textwidth,
				height=0.48 \textwidth,
				]
				
				\addplot[
				color=blue,
				mark=square*,
				thick
				]
				coordinates {
					(5, 0.4131)
					(6, 0.6824)
					(7, 1.4921)
					(8, 3.3134)
					(9, 9.8892)
				};
				\addlegendentry{SIROCCO}
				
				\addplot[
				color=red,
				mark=triangle*,
				thick
				]
				coordinates {
					(5, 0.0508)
					(6, 0.0536)
					(7, 0.0853)
					(8, 0.1254)
					(9, 0.1923)
				};
				\addlegendentry{LA-MOKA}
				
			\end{axis}
		\end{tikzpicture}
		\caption{Comparison of the average computation time between SIROCCO and LA-MOKA as the number of lines $n$ increases.}
		\label{fig:benchmark_graph}
	\end{figure}
	
	\begin{rmk}
		In~\cite[Sec.~4]{MarMar:Sirocco}, the authors also perform a benchmark to compare SIROCCO with an algorithm developed by Beltrán and Leykin~\cite{BelLey}. To implement it, they took random polynomials of degree ranging from $4$ to $14$. The runtime behavior as degree increases is far from what we observed in our benchmark. This is easily explained by their choice of polynomials: the probability of generating a line arrangement at random is zero. Therefore, their benchmark does not reflect the particular situation that we are considering in this paper.
	\end{rmk}

	\section{Proofs}\label{sec:proofs}
	
	In this section, we detail the proofs of two technical facts used in Section~\ref{sec:algo}. The first explains why Conditions~(\ref{eq:min/max_non_crossing_condition}) and~(\ref{eq:last_non_crossing_condition}) in Step~2 justify the approximation using the centers of the subsquares, while the second addresses how to handle multiple synchronous crossings in Step~3.
	
	\subsection{The condition on the $k$-grid}\label{subsec:condition}
	
	Let us recall the context, we have a segment $\ell_u$ that originating in the subsquare $T_u$ of $S_D$ with indices $g_D(L_u)$ and ending in the subsquare $T_u^i$ of $S_{D_i}$ with indices $g_{D_i}(L_u)$. We want to approximate $\ell_u$ with the segment $l_u$ starting and ending in the centers of same subsquares of $S_D$ and $S_{D_i}$. This can be done safely if for all $L_u, L_v \in \A$, and for all segments $d_u$ (resp. $d_v$) originating in $T_u$ (resp. $T_v$) and ending in $T_u^i$ (resp. $T_v^i$), we have $d_u \cap d_v = \emptyset$. This can be reformulated as follows: the intersection of the convex hull of $T_u$ and $T_u^i$ with the convex hull of $T_v$ and $T_v^i$ is empty.
	
	We now prove that satisfying at least one of the inequalities in Conditions~(\ref{eq:min/max_non_crossing_condition}) and~(\ref{eq:last_non_crossing_condition}) is equivalent to the disjointness of the two previous convex hulls. Up to an isometry, we can assume that $g_{D}(L_v) = g_{D_i}(L_v) = (0,0)$. In this setting, $g_{D}(L_u) = (a,b)$ and $g_{D_i}(L_u)=(c,d)$. Additionally, since initial and final points of $\ell_u$ and $\ell_v$ should be in the interior of the subsquares, we have to consider only the interior of the convex hulls.
	
	\medskip
	
	Let $H_u$ be the interior of the convex hull of $T_u$ and $T_u^i$. Its intersection with the plan $x = t$, for $t\in [0,1]$, is the interior of the unit square of the plane $x=t$ centered at $(t, (1-t)a + tc, (1-t)b + td)$. Let $H_v$ be the interior of the convex hull of $T_v$ and $T_v^i$, that is the open unit cube $(0,1)\times (-\frac{1}{2},\frac{1}{2})^2$. Its intersection with the plane $x = t$, for $t\in [0,1]$, is the interior of the unit square in the plane $x=t$ centered at $(t,0,0)$. The two solids overlap if and only if there exists $t\in [0,1]$ such that the distance between their centers in both the $y$- and $z$-coordinates is less than or equal to $1$. In other words, the two solids are disjoint if and only if, for all $t\in [0,1]$ we have:
	\begin{equation}\label{eq:disjoint_solids}
		1 \leq |(1-t)a + tc| \quad \text{or} \quad 1 \leq |(1-t)b + td|.
	\end{equation}
	
	\begin{theorem}\label{thm:conditions}
		The solids $H_u$ and $H_v$ are disjoint if and only if at least one of the following conditions holds:
		\begin{enumerate}
			\item $\min(a, c) \geq 1$
			\item $\max(a, c) \leq -1$
			\item $\min(b, d) \geq 1$
			\item $\max(b, d) \leq -1$
			\item $|ad - bc| \geq |a - c| + |b - d|$
		\end{enumerate}
	\end{theorem}
	
	\begin{proof}
		Let $(f_1(t), f_2(t)) = ( (1-t)a + tc, (1-t)b + td )$ be the $y$ and $z$ coordinates of the center of the intersection of $H_u$ with the plane $x=t$, the $x$ one being obviously $t$.
		
		\medskip
		
		If one of the conditions~(1) to~(4) is verified, it is direct to check that one of the two inequalities in~(\ref{eq:disjoint_solids}) is verified. For example, if $\min(a,c) \geq 1$, then $f_1(t) \geq (1-t) + t \geq 1$. Assume that the fifth condition $|ad - bc| \geq |a - c| + |b - d|$ holds. By contradiction, we assume that there exists $t\in[0,1]$ such that both of the inequalities of~(\ref{eq:disjoint_solids}) are false, that is to say that $|f_1(t)|<1$ and $|f_2(t)|<1$. Remark that
		\begin{equation*}
			ad - bc = f_1(t)(d-b) - f_2(t)(c-a).
		\end{equation*}
		Using this equality together with the triangle inequality, we get that
		\begin{equation*}
			|ad - bc| \leq |f(t)| \cdot |d - b| + |g(t)| \cdot |c - a|.
		\end{equation*}
		Due to the assumption that $|f_1(t)|<1$ and $|f_2(t)|<1$, we deduce that $|ad - bc| \leq |a - c| + |b - d|$ which explicitly contradicts the hypothesis that the fifth condition holds. So, such a $t\in [0,1]$ does not exist, and the solids $H_u$ and $H_v$ are disjoint.
		
		\medskip
		
		Let us prove the reverse implication by the contrapositive. We assume that none of the five conditions holds, and want to prove that the solids intersect. Up to symmetry, one can also assume that $c>a$ and $d>b$. Let $p = c - a$ and $q = d-b$, note that by assumption they are both strictly positive. Let $I_1$ and $I_2$ be the sets where the inequalities~(\ref{eq:disjoint_solids}) are not verified, that is to say:
		\begin{equation*}
			I_1 = \{ t\in[0,1] \mid |a-pt| < 1 \} \quad\text{and}\quad I_2 = \{ t\in[0,1] \mid |b-qt| < 1 \}.
		\end{equation*}
		Due to the assumption on $a$, $b$, $c$ and $d$, these sets can be described explicitly by:
		\begin{equation*}
			I_1 = \left( \max \left( 0, \frac{-1-a}{p} \right), \min \left( 1, \frac{1-a}{p} \right) \right)
			\quad\text{and}\quad
			I_2 = \left( \max \left( 0, \frac{-1-b}{q} \right), \min \left( 1, \frac{1-b}{q} \right) \right).
		\end{equation*}
		Since the four first conditions fails, we can easily check that $I_1 \neq \emptyset$ and $I_2 \neq \emptyset$. We claim that these intervals overlap. By contradiction, assume that they are disjoint. Without loss of generality we can assume that $I_1$ is at the left of $I_2$, that is the upper bound of $I_1$ is smaller than the lower bound of $I_2$. Since both are not empty and that they are assume to be disjoint, then the upper bound of $I_1$ cannot be $1$ and the lower bound of $I_2$ cannot be $0$. So, we have the inequality:
		\begin{equation*}
			\frac{1-a}{p} \leq \frac{-1-b}{q}.
		\end{equation*}
		Since $p$ and $q$ are strictly positive, we can rearrange the terms as:
		\begin{equation*}
			p+q \leq qa - pb.
		\end{equation*}
		Remark that $qa - pb = ad - bc$, so we have that $ad - bc \geq p+q = |c-a| + |d-b|$. This implies that $|ad-bc| \geq |a-c| + |b-d|$, which contradicts the assumption that Condition~(5) fails. Therefore, $I_1 \cap I_2 \neq \emptyset$. Consider any $t\in I_1 \cap I_2$. By definition, we have: $0 \leq t \leq 1$, $|(1-t)a + tc| < 1$ and $|(1-t)b + td| < 1$. By the inequalities~(\ref{eq:disjoint_solids}), we deduce that the solids $H_u$ and $H_v$ intersect, which complete the proof.
	\end{proof}
	
	We claimed at the end of Step~2 that for sufficiently large $k$, $k$-grids satisfying Conditions~(\ref{eq:min/max_non_crossing_condition}) or~(\ref{eq:last_non_crossing_condition}) always exist. Why is that true? We can think about the convex hulls $H_u$ and $H_v$ as open tubular neighborhoods of the segments $\ell_u^i$ and $\ell_v^i$ respectively with square cross-section. Since the non-crossing condition between two segments is Zariski-closed, then for any small enough deformation of $\ell_u^i$ and $\ell_v^i$ the non-crossing condition will be respected. In other words, if we replace $\ell_u^i$ and $\ell_v^i$ with two segments $\tilde{\ell}_u^i$ and $\tilde{\ell}_v^i$ contained in their respective tubular neighborhoods, then the braid formed by these two segments $\tilde{\ell}_u^i$ and $\tilde{\ell}_v^i$ is isotopic to the one form by $\ell_u^i$ and $\ell_v^i$. But, any tubular neighborhood of $\ell_u^i$ (resp. $\ell_v^i$) contains an $H_u$ (resp. $H_v$) for a sufficiently small square section. Obviously, $H_u$ and $H_v$ are disjoint and thus verifies the conditions given in Theorem~\ref{thm:conditions}.
	
	\begin{rmk}
		The previous paragraph shows that the necessary and sufficient condition in Theorem~\ref{thm:conditions} is needed to assure the existence of the good $k$-grids at Step~2.
	\end{rmk}

	\subsection{Synchronous crossings}\label{subsec:multiple_intersection}
	
	When describing the braid $b_i$ in Step~3, we noted that multiple pairs of segments might intersect simultaneously. The claimed solution to this problem is to proceed according to adjacency. Before proving that it actually works, let's make an example.

	\subsubsection{Example} Assume that we have three lines $L_u$, $L_v$ and $L_w$ such that $g_D(L_u)= (3,1)$, $g_D(L_v) = (2,2)$, $g_D(L_w)=(1,3)$, and $g_{D_i}(L_u) = (1,1)$, $g_{D_i}(L_v) = (2,2)$, $g_{D_i}(L_w)=(3,3)$. Thus the vector $\nu$ is $(w,v,u)$, the three pairs $(u,v)$, $(u,w)$ and $(v,w)$ create a crossing at $t = 1/2$.
	\begin{enumerate}
		\item In the vector $\nu$, only the pairs $(v,w)$ and $(u,v)$ are adjacent, so we consider one of them first. Let's take $(v,w)$. The index of $w$ (resp. $v$) in $\nu$ is $1$ (resp. $2$). Since the index of $w$ in $\nu$ is smaller than that one of $v$, the first term of $b_i$ is $\sigma_{\min(1,2)}^{e_{w,v}}=\sigma_1^{-1}$. Then, we exchange the position of $v$ and $w$ in $\nu$. After this first step, we have: $b_i=\sigma_1^{-1}$ and $\nu = (v,w,u)$.
		
		\item Having handled the pair $(v,w)$, we still need to process the pairs $(u,v)$ and $(u,w)$. In $\nu=(v,w,u)$, only the pair $(u,w)$ is adjacent. The index of $u$ in $\nu$ is $3$ while the one of $w$ is $2$. Thus, we add to $b_i$ the term $\sigma_{\min(2,3)}^{e_{w,u}} = \sigma_2^{-1}$, then we exchange the position in $\nu$ of $u$ and $w$. At the end of that step, we have: $b_i = \sigma_1^{-1} \sigma_2^{-1}$ and $\nu = (v,u,w)$.
		
		\item The last pair to handle is $(u,v)$, and they are adjacent in $\nu$. The index of $u$ in $\nu$ is $2$ and the one of $v$ is $1$. So, we add to $b_i$ the term $\sigma_{\min(1,2)}^{e_{v,u}} = \sigma_1^{-1}$, and we exchange the position of $u$ and $v$ in $\nu$. Thus, we obtain $b_i = \sigma_1^{-1} \sigma_2^{-1}\sigma_1^{-1}$ and $\nu = (u,v,w)$.
	\end{enumerate}
	
	\begin{rmk}
		If in step (1) we select $(u,v)$ instead of $(v,w)$, we obtain $b_i = \sigma_2^{-1} \sigma_1^{-1}\sigma_2^{-1}$ and $\nu = (u,v,w)$. Note that this braid $b_i$ is equivalent to the one above.
	\end{rmk}

	\subsubsection{Braid-equivalence}
	
	First, let us recall that
	\begin{equation*}
		\BB_n = \langle \sigma_1, \dots , \sigma_{n-1} \mid \sigma_i \sigma_{i+1}\sigma_i = \sigma_{i+1} \sigma_i \sigma_{i+1}, \sigma_i\sigma_j = \sigma_j\sigma_i \rangle,
	\end{equation*}
	where $1\leq i \leq n-2$ in the first group of relations and $|i-j|\geq 2$ in the second one. Similarly, the symmetric group has the following presentation:
	\begin{equation*}
		S_n = \langle s_1, \dots , s_{n-1} \mid s_i s_{i+1} s_i = s_{i+1} s_i s_{i+1}, s_i s_j = s_j s_i, s_i^2 = 1 \rangle.
	\end{equation*}
	We have the usual morphism $\phi:\BB_n \rightarrow S_n$, defined by $\phi(\sigma_i) = s_i$. With the previous presentations, we have that $\BB_n$ is the Artin group associated to the Coxeter group $S_n$.
	
	\medskip
	
	Second, note that when a multiple crossing appears, it consists necessarily of $\frac{m(m-1)}{2}$ pairs. Indeed, if $(u,v)$ and $(u,w)$ intersect at time $t$, then so does $(v,w)$. Let $b$ and $b'$ denote two braids obtained by choosing different adjacency orderings. For both of them, the order of the strands is reversed, $\phi(b) = \phi(b') = w_0$, where $w_0$ is the longest element (reversal permutation) in $S_m$. This yields two words $\omega$ and $\omega'$ representing $w_0$. Moreover, these words are reduced because the length of $w_0$ in $S_n$ is $\frac{n(n-1)}{2}$. By Matsumoto theorem~\cite{Mat}, one can transform $\omega$ into $\omega'$ using only Artin braid relations: $s_i s_{i+1} s_i = s_{i+1} s_i s_{i+1}$ and $s_i s_j = s_j s_i$. Since for any given pair of strands, the sign of their crossing is identical in $b$ and $b'$, the transformation from $\omega$ to $\omega'$ induces a transformation from $b$ to $b'$. As a consequence, $b$ and $b'$ are braid-equivalent.
	
	\begin{rmk}
		An intuitive way to see why this adjacency procedure works is to apply a small perturbation to each strand in such a way that the crossing appears in the desired order.
	\end{rmk}

	
	\bibliographystyle{plain}
	\bibliography{Bibtex}
	
\end{document}